\documentclass[10pt]{amsart}
\usepackage{eurosym}
\usepackage[pagebackref,colorlinks=true]{hyperref}
\usepackage{verbatim}
\usepackage{eucal,url,amssymb,enumerate,amscd,}
\usepackage{amsfonts}
\usepackage{amsmath,amsthm,amssymb,amscd,enumerate,eucal,url}
\usepackage[margin=1in]{geometry}

\numberwithin{equation}{section}
\newtheorem{thrm}{Theorem}[section]
\newtheorem{lemma}[thrm]{Lemma}
\newtheorem{prop}[thrm]{Proposition}
\newtheorem{cor}[thrm]{Corollary}
\newtheorem{dfn}[thrm]{Definition}
\newtheorem{rmrk}[thrm]{Remark}

  \allowdisplaybreaks

\def\G{{\nabla}}

\def\g{{\tilde g}}
\def\e{{\tilde\eta}}
\def\o{{\tilde\omega}}
\def\n{{\nabla^{ngt}}}
\def\bn{{\bar\nabla}}
\def\ne{{\nabla_1^{ngt}}}
\def\nd{{\nabla_2^{ngt}}}
\def\nt{{\nabla_t^{ngt}}}
\def\O{{\tilde\Omega}}

\begin{document}

\begin{abstract}
We show that  a connection
with skew-symmetric torsion satisfying the Einstein metricity
condition exists on an almost contact metric manifold exactly when it is D-homothetic to a cosymplectic manifold. In  dimension five, we get that the existence of a connection with skew torsion  satisfying the Einstein metricity
condition is equivalent to the existence of a Sasaki-Einstein 5-manifold and vice versa, any Sasaki-Einstein 5-manifold generates a two parametric family of connections with skew torsion
 satisfying the Einstein metricity
condition. Formulas for the curvature and the Ricci tensors of these  connections are presented in terms of the Sasaki-Einstein SU(2) structures.
\end{abstract}

\title[Non-symmetric  Riemannian gravity and Sasaki-Einstein 5-manifolds]{Non-symmetric
Riemannian gravity  and Sasaki-Einstein 5-manifolds}
\date{\today}
\author{Stefan Ivanov}
\address[Ivanov]{University of Sofia "St. Kl. Ohridski"\\
Faculty of Mathematics and Informatics\\
Blvd. James Bourchier 5\\
1164 Sofia, Bulgaria; Institute of Mathematics and Informatics,
Bulgarian Academy of Sciences}
 \email{ivanovsp@fmi.uni-sofia.bg}
\author{Milan Zlatanovi\'c}
\address[Zlatanovi\'c]{Department of Mathematics, Faculty of Science and Mathematics, University of Ni\v s, Vi\v segradska 33, 18000 Ni\v s, Serbia}
\email{zlatmilan@yahoo.com}

\date{\today }
\maketitle \tableofcontents


\setcounter{tocdepth}{2}

\section{Introduction}

In this note we consider the geometry arising on an odd-dimensional manifold from the Einstein  gravitational theory on a non-symmetric (generalized) Riemannian manifold
 $(M^{2n+1}, G=g+F)$, where the generalized metric $G$ has non-degenerate symmetric part $g$ and a skew-symmetric part $F$ of $rank=2n$.

\subsection{Motivation from general relativity}

General relativity (GR) was developed by  Albert
Einstein in 1916 \cite{Ein1}  and contributed to by
many others after 1916. In GR the equation $ds^2=g_{ij}dx^idx^j,
\quad (g_{ij}=g_{ji})$
  is valid, where
$g_{ij}$ are functions of a point. In GR, which is a four dimensional space-time continuum, metric
properties depend on mass distribution.
 The magnitudes $g_{ij}$  are known as \emph{gravitational
 potential}.
Christoffel symbols,  commonly expressed by $\Gamma^k_{ij}$, play
the role of  magnitudes which determine the
gravitational force field. General relativity explains gravity as
{ the} curvature of space-time.

In  GR the metric tensor {  obeys} the Einstein equations
$R_{ij}-\frac 12 R g_{ij}=T_{ij}$,
where $R_{ij}$ is the Ricci tensor of the metric of
space-time, $R$ is the scalar curvature of the metric, and
$T_{ij}$ is the energy-momentum tensor of matter. In 1922,
Friedmann \cite{Fried} found a solution in which the universe may
expand or contract, and later  Lema\^{\i}tre \cite{Lema} derived a
solution for an expanding universe. However, Einstein believed
that the universe was apparently static, and since  static
cosmology was not supported by the general relativistic field
equations, he added the cosmological constant
$\Lambda$ to the field equations, which became
$R_{ij}-\frac 12 R g_{ij}+\Lambda g_{ij} =T_{ij}$.
From 1923 to the end of his life Einstein worked on
various variants of Unified Field Theory \cite{Ein}. This theory had the aim to
unite {the theory of gravitation}
 and the theory of electromagnetism.

Starting from 1950,
Einstein used the real non-symmetric basic tensor
$G$, sometimes called \emph{generalized Riemannian
metric/manifold}. In this theory  the symmetric part $g_{ij}$ of
the basic tensor $G_{ij} (G_{ij}=g_{ij}+F_{ij})$ is related to
gravitation, and the skew-symmetric one $F_{ij}$ to
electromagnetism.

 More recently the idea of a non-symmetric metric tensor appears in
Moffat's non-symmetric gravitational theory \cite{Mof}.

In Moffat's theory the skew-symmetric part of the metric tensor
represents a Proca field (massive Maxwell field) which is a part
of the gravitational interaction, contributing to the rotation of
galaxies.

While on a Riemannian space  the connection coefficients are
expressed { as functions} of the metric, $g_{ij}$, in
Einstein's works the connection between these magnitudes is
determined by the so called \emph{Einstein metricity
condition,} i.e. the non-symmetric metric tensor $G$ and the
connection components $\Gamma_{ij}^k$ are connected by
\begin{equation}\label{metein}
\frac{\partial G_{ij}}{\partial x^m}-\Gamma ^p_{im}G_{pj}- \Gamma
^p_{mj}G_{ip}=0.
\end{equation}
A generalized Riemannian manifold satisfying the Einstein
metricity condition \eqref{metein} is  called an NGT-space
\cite{Ein,LMof,Mof}.

The choice of a connection in NGT is  uniquely determined in terms of the structure tensors \cite{IZl}.
Special attention was paid  when the
torsion of the NGT connection  is totally skew-symmetric with respect to the symmetric
part $g$ of $G$ \cite{Eis} (NGTS connection for short). One reason for that comes from  supersymmetric
string theories and non-linear $\sigma$-models (see e.g.
\cite{Str,y4,BSethi,GMPW,GPap} and
references therein), as well as from { the theory of gravity} itself
\cite{Ham}.

In even dimensions with nondegenerate skew-symmetric part $F$ one arrives { at}  Nearly K\"ahler manifolds, namely,  an almost Hermitian manifold
is NGTS  exactly when it is a Nearly
K\"ahler manifold \cite[Theorem~3.3]{IZl}. In this case the NGTS connection coincides with the
Gray connection \cite{Gr1,Gr2,Gr3}, which is the
unique connection with skew-symmetric torsion preserving the
Nearly K\"ahler structure \cite{FI}.  Nearly K\"ahler manifolds
(called almost Tachibana spaces in \cite{Yano}) were developed by
 Gray \cite{Gr1,Gr2,Gr3} and have been intensively
studied since then in \cite{Ki,N1,MNS,N2,But,FIMU}. Nearly
K\"ahler manifolds also appear in supersymmetric
string theories (see e.g. \cite{Po1,PS,LNP,GLNP}). The first
complete and therefore compact  inhomogeneous examples of 6-dimensional Nearly K\"ahler
manifolds were presented recently in \cite{FosH}.

In  odd dimensions $2n+1$ with $rank F=2n$ one gets an almost contact metric manifold. It is shown in \cite{IZl} that  a connection
with skew-symmetric torsion satisfying the Einstein metricity
condition exists on an almost contact metric manifold exactly when it is  \emph{almost nearly cosymplectic} (the precise definition is given in Definition~\ref{defanc} below).

The aim of this note is to investigate the geometry of almost-nearly cosymplectic spaces. It is well known that nearly cosymplectic manifolds are the odd dimensional analog of Nearly K\"ahler spaces and that the trivial circle bundle over a Nearly K\"ahler space is nearly cosymplectic.

We { establish} a   relation between almost nearly cosymplectic spaces { and}  nearly cosymplectic spaces, namely we present a special D-homothetic deformation relating these two objets. Applying the structure theorem for nearly cosymplectic structures established by de Nicola-Dileo-Yudin in \cite{NDY} we present a structure theorem for almost nearly cosymplectic structures (Theorem~\ref{ancstructure}) and give a formula relating the curvature and the Ricci tensors of the almost nearly cosymplectic structure { to} the corresponding D-homothetic nearly cosymplectic structure.

In  dimension five we found a closed relation { between} almost nearly cosymplectic structures {  and} the 5-dimensional Sasaki-Einstein spaces. Namely, using  the fundamental observation connecting nearly cosymplectic 5-manifolds with the SU(2) structures and Sasaki-Einstein 5-manifolds due to Cappelletti-Montano-Dileo in \cite{CD}, we show in Theorem~\ref{asasncos} that any almost nearly cosymplectic 5-manifold is D-homothetically equivalent to a 5-dimensional Sasaki-Einstein space. Moreover, we show that a Sasaki-Einstein 5-manifold generates a two parametric family of almost nearly cosymplectic structures and therefore a two-parametric family of NGTS connections.
We give an explicit formula connecting these NGTS connections  with the Levi-Civita connection of the corresponding Sasaki-Einstein structure.  We express the curvature and Ricci tensor of these NGTS connections in terms of the curvature of the Sasaki-Einstein metric and the corresponding Sasaki-Einstein SU(2) structure.

On the other hand, by virtue of admitting real Killing spinors \cite{FK}, Sasakian-Einstein 5-manifolds admit supersymmetry and
have received a lot of attention in physics from the point of view of AdS/CFT correspondence (see e.g. \cite{GauSas,GMPW1,Sas51,BBC,BBC1,Sp}). The AdS/CFT correspondence relates
string theory on the product of anti-deSitter space with a compact Einstein space to quantum
field theory on the conformal boundary.
The renewed interest in these manifolds has to do with the so-called p-brane
solutions in superstring theory.  For example, the case of
D3-branes of string theory the relevant near horizon geometry is
that of a product of anti-deSitter space with  a Sasakian-Einstein 5-manifold. This led to a construction of a number of examples of  irregular compact Sasaki-Einstein 5-manifolds \cite{GauSas,GMPW1,Sp}.

 In this spirit, our results may help to establish a possible  relation between the NGT with skew-symmetric torsion,  supersymmetric string theories and quantum field theory.

\section{Preliminaries}

\subsection{Einstein metricity condition (NGT)}\label{ngt}
In his attempt to construct a unified field theory
(Non-symmetric Gravitational Theory, briefly NGT)  Einstein
\cite{Ein} considered a generalized Riemannian manifold $(G=g+F)$ with nondegenerate symmetric part $g$ and skew-symmetric part $F$ and
used the so-called metricity condition
\eqref{metein}, which can be written as follows $
XG(Y,Z)-G(\n _YX,Z)-G(Y,\n_XZ)=0$.

For { the} non-degenerate  symmetric part $g$  we have a (1,1) tensor $A$ given  by
$F(X,Y)=g(AX,Y)$.

The metricity condition
\eqref{metein} can be written in terms of the torsion $T(X,Y)=\n_XY-\n_YX-[X,Y]$ of the NGT connection $\n$  and the endomorphism $A$ in the form
\begin{equation}\label{metein1}
(\n_XG)(Y,Z)=-G(T(X,Y),Z) \quad \Leftrightarrow \quad
(\n_X(g+F))(Y,Z)=-T(X,Y,Z)+T(X,Y,AZ).
\end{equation}
A general solution for the connection $\n$ satisfying
\eqref{metein} is given in  terms of   $g,F,T$\cite{Hlav} (see
also \cite{Mof}) and in terms of $dF$ in \cite{IZl}.

\subsection{Almost contact metric and almost nearly cosymplectic structures.}
In the case of  odd dimension, we consider an
almost contact metric manifold $(M^{2n+1}, g, \phi, \eta,\xi)$, i.e.,
a $(2n+1)$ -dimensional Riemannian manifold equipped with a 1-form
$\eta$, a (1,1)-tensor $\phi$ and { the} vector field $\xi$ dual to $\eta$
with respect to the metric $g$, $\eta(\xi)=1, \eta(X)=g(X,\xi)$
such that the following compatibility conditions are satisfied
(see e.g. \cite{Blair})
\begin{equation}\label{acon}
\phi^2=-id+\eta\otimes\xi, \quad g(\phi X,\phi Y)=g(X,Y)-\eta(X)\eta(Y), \quad \phi\xi=\eta\phi=0.
\end{equation}
The fundamental 2-form is defined by $F(X,Y)=g(\phi X, Y)$.

Such a space can be considered as a generalized Riemannian manifold with $G=g+F,\quad A=\phi$ and in this case the skew-symmetric part $F$ is degenerate $F(\xi,.)=0$ and has $rank F=2n$.

\begin{dfn}\cite{IZl}\label{defanc}
An almost contact metric manifold $(M^{2n+1},\phi,g,\eta,\xi)$ is
said to be \emph{almost-nearly cosymplectic} if the
covariant derivative of the fundamental tensor $\phi$ with respect to the Levi-Civita connection $\nabla$ of the Riemannian metric $g$ satisfies the
following   condition
\begin{equation}\label{ff3f}
\begin{split}
g((\nabla_X\phi)Y,Z)=\frac13dF(X,Y,Z)+\frac13\eta(X)d\eta(Y,\phi Z)-\frac16\eta(Y)d\eta(Z,\phi X)-\frac16\eta(Z)d\eta(\phi X,Y)
\end{split}
\end{equation}
\end{dfn}

The following relations are also valid \cite[Section~3.5]{IZl}
\begin{equation}\label{ax}
\begin{split}
d\eta(X,Z)=dF(X,\phi Z,\xi)=dF(\phi X,Z,\xi),\quad d\eta(X,\xi)=0,\\
d\eta(\phi X,Z)=d\eta(X,\phi Z)=dF(\phi X,\phi Z,\xi)=-d F(X,Z,\xi)\\
d F(\phi X,\phi Y,\phi Z)+dF(X,Y,\phi Z)=\eta(X)d\eta(Y,Z)+\eta(Y)\eta(Z,X).
\end{split}
\end{equation}
and the vector field $\xi$ is a Killing vector field
\begin{equation}\label{kill1}
(\nabla_X\eta)Y=g(\nabla_X\xi,Y)=\frac12d\eta(X,Y),
\end{equation}
Almost nearly cosymplectic manifolds arise in a natural way from NGT.
 Namely, we have the following

\begin{thrm}\cite[Theorem~3.8]{IZl}\label{acnika}
Let $(M,\phi,g,F,\eta,\xi)$ be an almost contact metric manifold with
a fundamental 2-form F considered as a generalized Riemannian
manifold $(M,G)$ with a generalized Riemannian metric $G=g+F$.
Then $(M,G)$ satisfies the Einstein metricity condition
\eqref{metein} with a totally skew-symmetric torsion $T$ if and
only if it is almost-nearly cosymplectic, i.e. \eqref{ff3f}
holds.

The skew-symmetric torsion is determined by the condition
\begin{equation}\label{nk3ac}
T(X,Y,Z)=-\frac13dF(X,Y,Z)
\end{equation}
The connection is  { uniquely} determined by the formula
\begin{equation}\label{ngtcon}
g(\n_XY,Z)=g(\nabla_XY,Z)-\frac16dF(X,Y,Z)+\frac16\Big[\eta(X)d\eta(Y,Z)+\eta(Y)d\eta(X,Z)
\Big].
\end{equation}
The Einstein metricity condition has the form
$$(\n_XG)(Y,Z)=\frac13\Big[dF(X,Y,Z)-dF(X,Y,\phi Z)\Big].$$The
covariant derivative of  $g$ and $F$ are given by
\begin{equation*}
\begin{split}
(\n_Xg)(Y,Z)=\frac16\Big[\eta(Y)d\eta(Z,X)+\eta(Z)d\eta(Y,X)\Big];\\
 (\n_XF)(Y,Z)=\frac13\Big[dF(X,Y,Z)-dF(X,Y,\phi Z)\Big]-\frac16\Big[\eta(Y)d\eta(Z,X)+\eta(Z)d\eta(Y,X)\Big].
\end{split}
\end{equation*}
\end{thrm}
\subsection{Nearly cosymplectic manifolds}
 We recall  here the definition  of nearly cosymplectic structures together with their basic properties from \cite{CD,NDY}.

An almost contact metric structure is cosymplectic if the endomorphism $\phi$ is parallel with respect to the Levi-Civita connection, $\nabla\phi=0$ and it is Sasakian if it is normal and contact, $[\phi,\phi]+\eta\otimes\xi=0, d\eta=-2F$. In terms of the Levi-Civita connection the Sasakian condition has the form

$(\nabla_X\phi)Y=g(X,Y)\xi-\eta(Y)X.$ For further details on Sasakian and  cosymplectic manifolds, we {refer the reader} to \cite{Blair,BG,CDY}.

 An almost contact metric manifold is called nearly cosymplectic if \cite{Bl,BS} $(\nabla_X\phi) X=0$
which is also equivalent to the condition
\begin{equation}\label{ncfxx}
g((\nabla_X\phi)Y,Z)=\frac13dF(X,Y,Z)
\end{equation}
In this case the vector field $\xi$ is Killing, $\nabla_{\xi}\xi=\nabla_{\xi}\eta=0$  and
\begin{equation}\label{ncxixx}
d\eta(\phi X,Y)=d\eta(X,\phi Y).
\end{equation}
The tensor field $h$ of type (1,1) defined by
\begin{equation}\label{hh}
\nabla_X\xi=hX
\end{equation}
has the following properties
\begin{eqnarray}\label{ncosh}
\aligned
g(hX,Y)=-g(X,hY)=\frac12d\eta(X,Y); \quad Ah+hA=0,
\endaligned\end{eqnarray}
i.e. it is skew-symmetric, anticommutes with $\phi$ and satisfies $h\xi=0, \eta\circ h=0$.

The following formulas also hold \cite{Endo,Endo1}:
\begin{equation}\label{nc15}
g((\nabla_X\phi)Y,hZ)=\eta(Y)g(h^2X,\phi Z)-\eta(X)g(h^2Y,\phi Z).
\end{equation}
\begin{equation}\label{nc16}
(\nabla_Xh)Y=g(h^2X,Y)\xi-\eta(Y)h^2X.
\end{equation}
\begin{equation}\label{nc17}
tr(h^2)=constant.
\end{equation}
A consequence of \eqref{nc16} and \eqref{nc17} is that the eigenvalues of symmetric operator $h^2$ are constants \cite{CD,NDY}. A fundamental observation that if $0$ is a simple eigenvalue of $h^2$ then the nearly cosymplectic manifold is of dimension five was made by A. de Nicola, G. Dileo and I. Yudin  in \cite{NDY} which lead to their
 structure theorem:
\begin{thrm} \cite[Theorem~4.5]{NDY} \label{ncstructure}
Let $(M,\phi, \xi, \eta, g)$  be a nearly cosymplectic non-cosymplectic manifold of dimension $2n+1>5$. Then $M$ is locally isometric to one of the following
Riemannian products:
$$\mathcal R\times N^{2n}, \quad M^5\times N^{2n-4},$$
where $N^{2n}$ is a nearly K\"ahler non-K\"ahler manifold,  $N^{2n-4}$ is a nearly K\"ahler  manifold and $M^5$
is a nearly cosymplectic non-cosymplectic manifold of dimension five.

If the manifold $M$ is complete and simply connected the above isometry is global.
\end{thrm}
 Note also that the nearly K\"ahler factor can be further decomposed according to \cite[Theorem~1.1 and Proposition~2.1]{N1}.

\subsection{Nearly cosymplectic manifolds in dimension 5} According to Theorem~\ref{ncstructure}, the non-trivial case of nearly cosymplectic manifolds is the 5-dimensional case. In this case B. Cappelletti-Montano and G. Dileo show in \cite{CD} that any 5-dimensional nearly cosymplectic manifold carryies a Sasaki-Einstein structure and vice-versa.
In order to describe the construction in \cite{CD} we recall below the notion of {an} $SU(2)$ structure developed by D. Conti and S. Salamon in \cite{CS}.

An SU(2) structure in dimension 5  is an SU(2)-reduction of the bundle  of linear frames and it  is equivalent to the existence of three almost contact
metric structures $(\phi_i, \xi, \eta, g), i=1,2,3$ related {to each other through} $\phi_i\phi_j=-\phi_j\phi_i=\phi_k$ for any even permutation \{i,j,k\} of \{1,2,3\}. In \cite{CS}
 Conti and Salamon proved that, in the spirit of special geometries, such a structure is equivalently determined by a quadruplet
$(\eta,\omega_1,\omega_2,\omega_3)$,
where $\eta$ is a 1-form and $\omega_i$ are 3 {2-forms} satisfying $\omega_i\wedge\omega_j=\delta_{ij}v$ for some 4-form $v$ with $v\wedge\eta\not=0$ and $X\lrcorner\omega_i=Y\lrcorner\omega_j\Longrightarrow \omega_k(X,Y)\ge 0$. The endomorphisms $\phi_i$, the Riemannian metric $g$ and the 2-forms $\omega_i$ are related {to each other through} $\omega_i(X,Y)=g(\phi X,Y)$. The class of Sasaki-Einstein structures in dimension 5 is characterized by the following differential equations
\begin{equation}\label{sasEinstein}
d\eta=-2\omega_3, \quad d\omega_1=3\eta\wedge\omega_2,\quad d\omega_2=-3\eta\wedge\omega_1.
\end{equation}
For such a manifold the almost contact metric structure $(\phi_3, \xi, \eta, g)$ is Sasakian, with Einstein
Riemannian metric $g$. A Sasaki-Einstein 5-manifold may { equivalently be} defined as a Riemannian manifold for which the cone metric is K\"ahler Ricci flat \cite{BG}.

There are several generalizations of Sasaki-Einstein structures in dimension five. We only recall that the  hypo structures introduced in \cite{CS} are defined by
\begin{equation}\label{hypo}
d\omega_3=0, \quad d(\eta\wedge\omega_1)=0,\quad d(\eta\wedge\omega_2)=0.
\end{equation}
These structures arise naturally on hypersurfaces of a 6-manifold endowed with an integrable SU(3) structure, i.e. a hypersurface of a K\"ahler Ricci flat 6-manifold.

Starting with a 5 dimensional nearly cosymplectic manifold  B. Cappelletti-Montano and G. Dileo show in \cite[Theorem~5.1]{CD} that
\begin{equation}\label{spech}h^2=-\lambda^2(I-\eta\otimes\xi), \lambda=const.\not=0
\end{equation}
induces an SU(2) structure $(\eta,\omega_1,\omega_2,\omega_3)$ determined by
\begin{equation}\label{su2}
\phi_1=-\frac1{\lambda}\phi h, \quad \phi_2=\phi, \quad \phi_3=-\frac1{\lambda}h; \qquad \omega_i(X,Y)=g(\phi_i X,Y)
\end{equation}
which satisfies the relations
\begin{equation}\label{sasnk}
d\eta=-2\lambda\omega_3, \quad d\omega_1=3\lambda\eta\wedge\omega_2,\quad d\omega_2=-3\lambda\eta\wedge\omega_1.
\end{equation}
In particular these SU(2) structures are hypo.

Consider the homothetic  SU(2) structures determined by $\tilde\eta=\lambda\eta, \quad \tilde\omega_i=\lambda^2\omega_i$ It follows from  \eqref{sasnk} that these new SU(2) structures satisfy \eqref{sasEinstein} and therefore $\o_3$ is a Sasaki-Einsten structure while $\o_1$ and $\o_2$ are nearly cosymplectic structures \cite{CD}. This shows
\begin{thrm}\cite{CD}\label{sasncos}
A nearly cosymplectic non cosymplectic 5-manifold carries a Sasaki-Einstein structure and vice versa, any Sasaki-Einstein 5 manifold supports 2 nearly cosymplectic structures. In particular, a nearly cosymplectic non cosymplectic 5-manifold is  Einstein with positive scalar curvature.

In terms of nearly cosymplectic structures, the attached  Sasaki-Einstein structure
 is given by
$$
\tilde\phi=-\frac1{\lambda}h,\quad \tilde\xi=\frac1{\lambda}\xi,\quad \tilde\eta=\lambda\eta,\quad \tilde g=\lambda^2g,\quad Scal=\lambda^2 \tilde Scal=20\lambda^2
$$
while $(\phi,\tilde\eta,\tilde g)$ and $(-\frac1{\lambda}\phi h,\tilde\eta,\tilde g)$ are nearly cosymplectic structures.
\end{thrm}

\section{Almost nearly cosymplectic manifolds}
Let $(M,\phi, \xi, \eta, g)$  be an almost  nearly cosymplectic  manifold of dimension $2n+1$. For the (1,1) tensor $h$ defined by \eqref{hh} we have
\begin{lemma}
On an almost nearly cosymplectic manifolds the following relations { are}
valid:
\begin{eqnarray}\label{ancosh}
\aligned
&g(hX,Y)=-g(X,hY)=\frac12d\eta(X,Y);\\
&h\xi=0,\quad \phi h+h\phi=0;\quad (\nabla_X\phi)\xi=-\phi hX.
\endaligned\end{eqnarray}
\end{lemma}
\begin{proof}
The Killing condition \eqref{kill1} together with \eqref{hh} yield
$g(\nabla_X\xi,Y)=g(hX,Y)=\frac12d\eta(X,Y)=-g(hY,X)$ which proves the first equality.  The second and the third are consequences of the first one and the first two lines in \eqref{ax}. Indeed, for the third one  we have
$$g(h\phi X,Y)=\frac12d\eta(\phi X,Y)=\frac12d\eta(X,\phi Y)=g(hX,\phi Y)=-g(\phi hX,Y)$$
 The last equation follows directly from \eqref{ff3f}.
\end{proof}

\subsection{D-homothetic transformations}
We recall \cite{Tano, Tano1} that the almost contact metric structure $(\bar\phi,\bar\xi,\bar\eta,\bar g)$ defined by
\begin{equation}\label{dhom}
\bar\eta=a\eta, \quad \bar\xi=\frac1a\xi, \quad \bar\phi=\phi, \quad \bar g=ag+(a^2-a)\eta\otimes\eta, \quad a>0 \quad constant
\end{equation}
is called D-homothetic to $(\phi,\xi,\eta,g)$.

We have $\bar F=aF$.

Our main result follows
\begin{thrm}\label{d-hom}
Any almost nearly cosymplectic structure is D-homothetic to a nearly cosymplectic structure and vice versa.

The corresponding (1,1) tensors $\bar h$ and $h$ coincide.
\end{thrm}
\begin{proof}
Let $\{X_1,\dots,X_{2n},X_{2n+1}=\xi\}$ be an orthonormal basis. The Kozsul formula and \eqref{dhom} give
\begin{multline}\label{conn}
2\bar g(\bar\nabla_{X_i}X_j,X_k)=X_i(\bar g(X_j,X_k)+X_j(\bar g(X_i,X_k)-X_k(\bar g(X_i,X_j)\\+\bar g([X_i,X_j],X_k)+\bar g([X_k,X_i],X_j)-\bar g([X_j,X_k],X_i)
\\=2a g(\nabla_{X_i}X_j,X_k)
+(a^2-a)\Big[X_i[\eta(X_j)\eta(X_k)]+X_j[\eta(X_k)\eta(X_i)]-X_k[\eta(X_i)\eta(X_j)]\Big]\\
+(a^2-a)\Big[\eta([X_i,X_j])\eta(X_k)]+\eta([X_k,X_i])\eta(X_i)]-\eta([X_j,X_k])\eta(X_i)\Big]\\
=2a g(\nabla_{X_i}X_j,X_k)
+(a^2-a)\Big[(\nabla_{X_i}\eta)X_j+(\nabla_{X_j}\eta)X_i+2\eta(\nabla_{X_i}X_j)\eta(X_k)\Big]\\
+(a^2-a)\Big[[\nabla_{X_i}\eta)X_k-(\nabla_{X_k}\eta)X_i]\eta(X_j)+[\nabla_{X_j}\eta)X_k-(\nabla_{X_k}\eta)X_j]\eta(X_i)\Big]\\
=2a g(\nabla_{X_i}X_j,X_k)
+(a^2-a)\Big[(\nabla_{X_i}\eta)X_j+(\nabla_{X_j}\eta)X_i+2\eta(\nabla_{X_i}X_j)\eta(X_k)\Big]\\
+(a^2-a)\Big[d\eta(X_i,X_k)\eta(X_j)+d\eta(X_j,X_k)\eta(X_i)\Big]
\end{multline}
The Killing condition applied to \eqref{conn} yields
\begin{multline}\label{conn1}
\bar g(\bar\nabla_{X_i}X_j,X_k)=a g(\nabla_{X_i}X_j,X_k)\\
+\frac{a^2-a}2\Big[d\eta(X_i,X_k)\eta(X_j)+d\eta(X_j,X_k)\eta(X_i)+2\eta(\nabla_{X_i}X_j)\eta(X_k)\Big]
\end{multline}
Using the Killing condition, we evaluate the last term in \eqref{conn1} as follows,
\begin{equation}\label{term}
2\eta(\nabla_{X_i}X_j)=2g(\nabla_{X_i}X_j,\xi)=-2g(\nabla_{X_i}\xi,X_j)=-d\eta(X_i,X_j).
\end{equation}
For a D-homothetic transformation with Killing vector field $\xi$ we obtain substituting \eqref{term} into \eqref{conn1} that
\begin{multline}\label{conKil}
\bar g(\bar\nabla_{X_i}X_j,X_k)=ag(\nabla_{X_i}X_j,X_k)\\+\frac{a^2-a}2\Big[d\eta(X_i,X_k)\eta(X_j)+d\eta(X_j,X_k)\eta(X_i)-d\eta(X_i,X_j)\eta(X_k)\Big]\end{multline}
We have, applying \eqref{conKil}, that the following tensor equality holds for all vector fields $X,Y,Z$
\begin{multline}\label{barF}
\bar g((\bar\nabla_X\phi)Y,Z)=\bar g(\bar\nabla_X\phi Y,Z)+\bar g(\bar\nabla_XY,\phi Z)\\=ag((\nabla_X\phi)Y,Z)+\frac{a^2-a}2\Big[d\eta(\phi Y,Z)\eta(X)-d\eta(X,\phi Y)\eta(Z)+d\eta(X,\phi Z)\eta(Y)+d\eta(Y,\phi Z)\eta(X)\Big]
\\=ag((\nabla_X\phi)Y,Z)+\frac{a^2-a}2\Big[2d\eta(\phi Y,Z)\eta(X)+d\eta(X,\phi Z)\eta(Y)-d\eta(X,\phi Y)\eta(Z)\Big],
\end{multline}
where the last equality follows from \eqref{ncxixx}.

Suppose $(\phi,\xi,\eta,g)$ is a nearly cocymplectic structure, i.e. \eqref{ncfxx} holds. Then, taking $a=\frac32$,
the { equalities} \eqref{barF}, \eqref{ncfxx} and \eqref{ncxixx}  { imply}
\begin{multline}\label{anc-nc}
\bar g((\bar\nabla_X\phi)Y,Z)=\frac32g((\nabla_X\phi)Y,Z)+\frac38\Big[2d\eta(\phi Y,Z)\eta(X)+d\eta(X,\phi Z)\eta(Y)-d\eta(X,\phi Y)\eta(Z)\Big]\\
=\frac13d\bar F(X,Y,Z)+\frac16\Big[2d\bar\eta(Y,\phi Z)\bar\eta(X)-d\bar\eta(Z,\phi X)\bar\eta(Y)-d\bar\eta(\phi X,Y)\bar\eta(Z)\Big]
\end{multline}
i.e. the structure  $(\bar\phi,\bar\xi,\bar\eta,\bar g)$ is an almost nearly cosymplectic. Vice versa, starting from an almost nearly cosymplectic structure and making a D-homothetic deformation with constant $a=\frac23$ we get a nearly cosymplectic one which proves the first claim.

To show that $\bar h=h$,  we put $X_j=\bar \xi=\frac23\xi$ in \eqref{conKil} taken  with $a=\frac32$  and use \eqref{ancosh} to get
\begin{equation*}
\bar g(\bar\nabla_{X_i}\bar \xi,X_k)=g(\nabla_{X_i}\xi,X_k)+\frac{1}4d\eta(X_i,X_k)=\frac{3}2g(\nabla_{X_i}\xi,X_k).\end{equation*}
On the other { hand}, from (\ref{dhom}), we have
\begin{equation*}
\bar g(\nabla_{X_i} \xi,X_k)=\frac 3 2g( \nabla_{X_i}
\xi,X_k)+\frac{3}4\eta( \nabla_{X_i} \xi)\eta(X_k)=\frac 3 2g(\nabla_{X_i}\xi,X_k).
\end{equation*}
The last two equalities imply $ \bar hX_i=hX_i$ which completes the proof .
\end{proof}
\begin{cor}\label{nabla}
Let $(M,\phi, \xi, \eta, g)$  be an almost contact metric manifold with Killing vector field $\xi$. The D-homothetic almost contact metric manifold  $(M,\bar\phi, \bar\xi, \bar\eta, \bar g)$
has a Killing vector field $\bar\xi$ and { the two corresponding Levi-Civita connections $\bn$ and $\nabla$ of these manifolds} are related by
\begin{equation}\label{d-nabla}
g(\bar\nabla_XY,Z)=g(\nabla_XY,Z)+\frac{a^2-a}{2a}\Big[d\eta(X,Z)\eta(Y)+d\eta(Y,Z)\eta(X)\Big].
\end{equation}
\end{cor}
\begin{proof}
Observe that any orthonormal basis for $g$ is an orthogonal basis for $\bar g$. We obtain from \eqref{conKil} and \eqref{dhom}
$$
\bar g(\bar\nabla_{X_i}\bar\xi,X_k)=g(\nabla_{X_i}\xi,X_k)+\frac{a^2-a}{2a}d\eta(X_i,X_k)
$$
which shows that $\bar\xi$ is a Killing vector field for $\bar g$.

 Using \eqref{dhom} together with the Killing condition, we have
\begin{multline}\label{barnabla}
\bar g(\bar\nabla_{X_i}X_j,X_k)=ag(\bar\nabla_{X_i}X_j,X_k)+(a^2-a)\eta((\bar\nabla_{X_i}X_j)\eta(X_k)\\=ag(\bar\nabla_{X_i}X_j,X_k)+\frac{a^2-a}a\bar\eta((\bar\nabla_{X_i}X_j)\eta(X_k)=ag(\bar\nabla_{X_i}X_j,X_k)-\frac{a^2-a}{2a}d\bar\eta({X_i}X_j)\eta(X_k)\\
=ag(\bar\nabla_{X_i}X_j,X_k)-\frac{a^2-a}{2}d\eta({X_i}X_j)\eta(X_k)
\end{multline}
Compare \eqref{barnabla} with \eqref{conKil} to get
\begin{equation*}\label{nnabla}
ag(\bar\nabla_{X_i}X_j-\nabla_{X_i}X_j,X_k)=\frac{a^2-a}2\Big[d\eta(X_i,X_k)\eta(X_j)+d\eta(X_j,X_k)\eta(X_i)\Big]
\end{equation*}
which implies \eqref{nabla} since the difference between { these two connections} is a tensor field.
\end{proof}
Combine Theorem~\ref{d-hom} with Theorem~\ref{ncstructure} to get the structure theorem for almost nearly cosymplectic manifold.
\begin{thrm}\label{ancstructure}
Let $(M,\phi, \xi, \eta, g)$  be an almost nearly cosymplectic non-cosymplectic manifold of dimension $2n+1>5$. Then $M$ is locally D-homothetic with { the} constant $a=\frac23$ to one of the following
Riemannian products:
$$\mathcal R\times N^{2n}, \quad M^5\times N^{2n-4},$$
where $N^{2n}$ is a nearly K\"ahler non-K\"ahler manifold,  $N^{2n-4}$ is a nearly K\"ahler  manifold and $M^5$
is a nearly cosymplectic non-cosymplectic manifold of dimension five.

If the manifold $M$ is complete and simply connected the above D-homothety is global.
\end{thrm}
We also have
\begin{prop}
On an almost nearly cosymplectic manifolds $(M,\phi,\bar\eta,\bar\xi,\bar g)$ the { following} relations hold.
\begin{equation}\label{acn15}
\bar g((\bar\nabla_X\phi)Y,\bar hZ)=\bar\eta(Y)\bar g(\bar h^2X,\phi Z),
\end{equation}
\begin{equation}\label{acn16}-\bar R(\bar\xi, X,Y,Z)=
\bar g((\bar\nabla_X\bar h)Y,Z)=-\bar\eta(Y)\bar
g(\bar h^2X,Z)+\bar\eta(Z)\bar g(\bar h^2X,Y).
\end{equation}
\end{prop}
\begin{proof} We apply Theorem~\ref{d-hom}.  Using   (\ref{anc-nc}) and $\bar h=h$, we have
\begin{equation*}
\bar g((\bar\nabla_X\phi)Y,\bar hZ)=\frac{3}{2}g((\nabla_X\phi)Y, h Z)
+\frac{3}8\Big[2d\eta(\phi Y, hZ)\eta(X)+d\eta(X,\phi hZ)\eta(Y)\Big]
\end{equation*}
 For { this} D-homothetic transformation, and formula (\ref{nc15}), we
obtain
\begin{multline*}
\bar
g((\bar\nabla_X\phi)Y,\bar hZ)=\frac{3}{2}\Big[\eta(Y)g(h^2X,\phi Z)-\eta(X)g(h^2Y,\phi Z)\Big]
+\frac{3}8\Big[4g(h^2Y,\phi Z)\eta(X)+2g(h^2X,\phi Z)\eta(Y)\Big]\\=\frac{9}{4}\eta(Y)g(h^2X,\phi Z)=\bar\eta(Y)\bar g(\bar h^2X,\phi Z)
\end{multline*}
which proves \eqref{acn15}.

To prove the second line we first note that  since $\bar\xi$ is a Killing vector field we have the well known relation  $g((\bar\nabla_X\bar h)Y,Z)=-\bar R(\bar\xi, X,Y,Z)$ (see e.g. \cite{YanoBochner}).
Further, starting from
(\ref{d-nabla}) taken for $a=\frac32$, we obtain
\begin{equation*}
\bar
g((\bar\nabla_X\bar h)Y,Z)=\frac{3}{2}g((\nabla_Xh)Y,Z)+\frac{3}8\Big[d\eta(X,hZ)\eta(Y)
-d\eta(X,hY)\eta(Z)\Big],
\end{equation*}
which, in view of  (\ref{nc16}) and \eqref{ncosh}, takes the form
\begin{multline*} \bar
g((\bar\nabla_X\bar h)Y,Z)=\frac{3}{2}\Big[-\eta(Y)g(h^2X,Z)+\eta(Z)g(h^2X,Y)\Big]
+\frac{3}8\Big[-d\eta(hX,Z)\eta(Y) +d\eta(hX,Y)\eta(Z)\Big]\\
=-\frac{9}{4}\eta(Y)g(h^2X,Z)+\frac{9}{4}\eta(Z)g(h^2X,Y)=-\bar\eta(Y)\bar g(h^2X,Z)+\bar\eta(Z)\bar g(h^2X,Y).
\end{multline*}
This completes the proof.
\end{proof}
\subsection{The curvature of almost nearly cosymplectic manifold}
Let  $(M,\bar\phi,\bar\eta,\bar\xi,\bar g)$ be an almost nearly cosymplectic manifold D-homotetically related with $a=\frac32$ {  to} the nearly cosymplectic manifold  $(M,\phi,\eta,\xi, g)$ in the sense of Theorem~\ref{d-hom}.

Put $a=\frac32$ into \eqref{d-nabla} and use \eqref{ncosh} to get the following relations between the Levi-Civita connections
\begin{equation}\label{con12}
\bar\nabla_XY=\nabla_XY+\frac12\eta(Y)hX+\frac12\eta(X)hY
\end{equation}
For the curvatures, we calculate using \eqref{con12}, \eqref{ancosh}, \eqref{ncosh} and $\bar h=h$  that
\begin{multline}\label{curv}
\bar R(Z,X)Y=\bar\nabla_Z\bar\nabla_XY-\bar\nabla_X\bar\nabla_ZY-\bar\nabla_{[Z,X]}Y\\=R(Z,X)Y+\frac12(\nabla_Z\eta)Y.hX-\frac12(\nabla_X\eta)Y.hZ+\frac12d\eta(Z,X)hY
-\frac12\eta(Z)(\nabla_Xh)Y\\+\frac12\eta(X)(\nabla_Zh)Y+\frac12\eta(Y)\Big[(\nabla_Zh)X-(\nabla_Xh)Z\Big]+\frac14\eta(Y)\eta(Z)h^2X-\frac14\eta(Y)\eta(X)h^2Z.
\end{multline}
Applying  \eqref{nc16} to \eqref{curv}, we get
\begin{multline}\label{curv1}
\bar R(Z,X)Y=
R(Z,X)Y+\frac12g(hZ,Y)hX-\frac12g(hX,Y)hZ+g(hZ,X)hY\\+\frac54\eta(Y)\eta(Z)h^2X-\frac54\eta(Y)\eta(X)h^2Z+\frac12\Big[\eta(X)g(h^2Z,Y)-\eta(Z)g(h^2X,Y)\Big]\xi
\end{multline}
For the Ricci tensors $\overline{Ric}$ and $Ric$, we get
\begin{equation}\label{ric}
\overline{Ric}(X,Y)=Ric(X,Y)+g(h^2X,Y)-\frac54tr(h^2)\eta(X)\eta(Y).
\end{equation}

\section{Almost nearly cosymplectic manifolds in dimension 5}
In view of Theorem~\ref{ancstructure} we restrict our attention {to} dimension five.
Let $(M, \bar\phi, \bar\xi, \bar\eta,\bar g)$ be a 5-dimensional almost nearly
cosymplectic manifold, $(M, \phi, \xi, \eta, g)$ be the D-homothetically corresponding nearly cosymplectic manifold according to Theorem~\ref{d-hom} and
$(M, \tilde\phi, \tilde\xi, \tilde\eta, \tilde g)$ be the Sasaki-Einstein manifold homothetically attached to the nearly cosymplectic structure  $(M, \phi, \xi, \eta, g)$,  i.e. we have
\begin{equation}\label{ndhom}
\begin{split}
\bar\eta=\frac32\eta, \quad \bar\xi=\frac23\xi, \quad \bar\phi=\phi, \quad \bar g=\frac32g+\frac34\eta\otimes\eta, \quad g=\frac23\bar g-\frac29\bar\eta\otimes\bar\eta,\\
\tilde\phi=-\frac1{\lambda}h, \quad\tilde\eta=\lambda\eta=\frac{2\lambda}3\bar\eta, \quad \tilde\xi=\frac1{\lambda}\xi=\frac3{2\lambda}\bar\xi, \quad\tilde g=\lambda^2g=\frac{2\lambda^2}3\bar g-\frac{2\lambda^2}9\bar\eta\otimes\bar\eta
\end{split}
\end{equation}
We recall that an almost contact metric manifold  $(M, \phi, \xi, \eta, g)$ is called $\eta$-Einstein if its Ricci tensor satisfies
$$Ric(X,Y)=ag(X,Y)+b\eta(X)\eta(Y),$$
where $a$ and $b$ are smooth functions on $M$. $\eta$-Einstein metrics
 were introduced and studied by Okumura \cite{Ok}. In particular, he studied the relation between the existence of
$\eta$ -Einstein metrics and certain harmonic forms and {showed} that a K-contact $\eta$-Einstein manifold of dimension $2n+1$ bigger than 3 the functions $a$ and $b$ are constants satisfying $a+b=2n$. Sasakian $\eta$-Einsten spaces are studied in detail in \cite{BGM}.
\begin{prop}\label{cur}
Let $(M, \bar\phi, \bar\xi, \bar\eta,\bar g)$ be a 5-dimensional almost nearly
cosymplectic manifold. Then it is  $\eta$-Einstein and the  Ricci tensor is given by
\begin{equation}
\overline{Ric}(X,Y)=2\lambda^2\bar g(X,Y)+2\lambda^2\bar\eta(X)\bar\eta(Y).
\end{equation}
\end{prop}
\begin{proof}
 We get from \eqref{curv1} and \eqref{spech}
\begin{multline}\label{curv5}
\bar R(Z,X)Y=
R(Z,X)Y+\frac12g(hZ,Y)hX-\frac12g(hX,Y)hZ+g(hZ,X)hY\\-\frac{5\lambda^2}4\eta(Y)\eta(Z)X+\frac{5\lambda^2}4\eta(Y)\eta(X)Z-\frac{\lambda^2}2\Big[\eta(X)g(Z,Y)-\eta(Z)g(X,Y)\Big]\xi
\end{multline}
For the Ricci tensors $\overline{Ric}$ and $Ric$, we obtain from \eqref{ric}
\begin{multline}\label{ric1}
\overline{Ric}(X,Y)=Ric(X,Y)+g(h^2X,Y)-\frac54tr(h^2)\eta(X)\eta(Y)\\=Ric(X,Y)-\lambda^2g(X,Y)+6\lambda^2\eta(X)\eta(Y)= 3\lambda^2 g+6\lambda^2\eta(X)\eta(Y)\\=3\lambda^2(\frac23\bar g(X,Y)-\frac29\bar\eta(X)\bar\eta(Y))+6\lambda^2\frac49\bar\eta(X)\bar\eta(Y)=2\lambda^2\bar g(X,Y)+2\lambda^2\bar\eta(X)\bar\eta(Y)
\end{multline}
since $Ric=\frac{Scal}5g=4\lambda^2g$ is an Einstein space according to Theorem~\ref{sasncos}. This completes the proof.
\end{proof}

Let $(\eta,\omega_1,\omega_2,\omega_3)$ be the SU(2) structure induced by the nearly cosymplectic structure determined by \eqref{su2} and satisfying \eqref{sasnk}  \cite{CD}. Since $\bar\phi=\phi$ and $\bar h=h$ we get an SU(2) structure $(\bar\eta,\bar\omega_1,\bar\omega_2,\bar\omega_3, \bar\omega_i=\bar g(\phi_i.,.))$   induced by the almost nearly cosymplectic structure. These two SU(2) structures are related by
\begin{equation}\label{su22}
\bar\eta=\frac32\eta,\quad \bar\omega_i=\frac32\omega_i.
\end{equation}
We obtain from \eqref{sasnk}, Theorem~\ref{sasncos}, \eqref{su22} and Proposition~\ref{cur} the following
\begin{prop}\label{ancdf} An almost nearly cosymplectic structure on a 5-dimensional manifold is
equivalent to an $SU(2)$-structure
$(\bar\eta,\bar\omega_1,\bar\omega_2,\bar\omega_3)$ satisfying
\begin{equation}\label{achom}
d\bar\eta=-2\lambda\bar\omega_3,\quad d\bar\omega_1=2\lambda \bar\eta \wedge \bar\omega_2,\quad d\bar\omega_2=-2\lambda \bar\eta \wedge \bar\omega_1
\end{equation}
for some real number $\lambda\neq0$.
These structures are hypo and the structures $(\bar\eta,\bar\omega_1)$ and $(\bar\eta,\bar\omega_2)$ are almost nearly cosymplectic $\bar\eta$-Einstein structures.
\end{prop}
The homothetic $SU(2)$ structure $\eta^*=\lambda^2\bar\eta, \omega_i^*=\lambda^2\bar\omega_i, i=1,2,3$ satisfies \eqref{achom} with $\lambda=1$. The structure $(\eta^*,\omega^*_3)$ is a Sasaki $\eta^*$-Einstein structure while the structures  $(\eta^*,\omega^*_1)$ and  $(\eta^*,\omega^*_2)$ are almost nearly cosymplectic $\eta^*$-Einstein structures.
We have
\begin{cor}
An SU(2) structure  $(\eta,\omega_1,\omega_2,\omega_3)$ on a 5-manifold is Sasaki $\eta$-Einstein if and only if it satisfies the relations
\begin{equation}\label{achoms}
d\eta=-2\omega_3,\quad d\omega_1=2\eta \wedge \omega_2,\quad d\omega_2=-2\eta \wedge \omega_1.
\end{equation}
\end{cor}
Consider the structures $(\bar\eta,\bar\Omega^t_1,\bar\Omega^t_2,\omega_3)$, where
\begin{equation}\label{circ}
\begin{split}
\bar\Omega^t_1=\cos t\bar\omega_1+\sin t\bar\omega_2, \quad \phi_1^t=-\frac1{\lambda}\phi h\cos t+\phi\sin t\\
\bar\Omega^t_2=-\sin t\bar\omega_1+\cos t\bar\omega_2, \quad \phi_2^t=\frac1{\lambda}\phi h\sin t+\phi\cos t.
\end{split}
\end{equation}
It is easy  to check that these structures are SU(2) structures satisfying \eqref{achom} and  $(\bar\eta,\bar\Omega^t_1)$ and  $(\bar\eta,\bar\Omega^t_2)$ are almost nearly cosymplectic $\bar\eta$-Einstein structures.
\begin{rmrk}\label{rmnk}
Starting with an SU(2) structure  $(\eta,\omega_1,\omega_2,\omega_3)$ on a 5-manifold which  is Sasaki Einstein one obtaines an $(\mathbb R^2-\{0\})$-familly of  nearly cosymplectic structures $(\eta,\Omega^t_1,\Omega^t_2,\omega_3)$ given by
\begin{equation}\label{circnc}
\begin{split}
\Omega^t_1=\cos t\omega_1+\sin t\omega_2, \quad \phi_1^t=-\frac1{\lambda}\phi h\cos t+\phi\sin t\\
\Omega^t_2=-\sin t\omega_1+\cos t\omega_2, \quad \phi_2^t=\frac1{\lambda}\phi h\sin t+\phi\cos t
\end{split}
\end{equation}
which satisfy \eqref{sasnk}.
\end{rmrk}

Applying Theorem~\ref{sasncos}, Theorem~\ref{d-hom} and Proposition~\ref{cur}, we obtain
\begin{thrm}\label{asasncos}
An almost nearly cosymplectic non cosymplectic 5-manifold  $(M, \bar\phi, \bar\xi, \bar\eta,\bar g)$ carries a D-homothetic Sasaki-Einstein structure $(M, \tilde\phi, \tilde\xi, \tilde\eta, \tilde g)$ and vice versa, any Sasaki-Einstein 5 manifold supports a $(\mathbb R^2-\{0\})$-familly of almost nearly cosymplectic structures.

 In particular, an almost  nearly cosymplectic non cosymplectic 5-manifold is  $\eta$-Einstein, $\overline{Ric}=2\lambda^2(\bar g+\bar\eta\otimes\bar\eta)$  with positive scalar curvature $\overline{Scal}=12\lambda^2$.

In terms of  almost nearly cosymplectic structures, the attached  Sasaki-Einstein structure 
is given by
$$
\tilde\phi=-\frac1{\lambda}h, \quad\tilde\eta=\frac{2\lambda}3\bar\eta, \quad \tilde\xi=\frac3{2\lambda}\bar\xi, \quad\tilde g=\frac{2\lambda^2}3\bar g-\frac{2\lambda^2}9\bar\eta\otimes\bar\eta
$$
The  structures $(-\frac1{\lambda}\phi h,\bar\eta,\bar g)$ and  $(\phi,\bar\eta,\bar g)$ are almost nearly cosymplectic  generating the circle-family  $(\phi_1^t,\bar\eta,\bar g), (\phi_2^t,\bar\eta,\bar g)$ of almost nearly cosymplectic structures  defined by \eqref{circ}.
\end{thrm}

Since the nearly cosymplectic structure is homothetic to the Sasaki-Einstein structure then the corresponding Levi-Civita connections coincide, $\nabla=\tilde\nabla$ and \eqref{con12} takes the form
\begin{equation}\label{aconsas}
\bar\nabla_XY=\tilde\nabla_XY+\frac12\eta(Y)hX+\frac12\eta(X)hY,\quad R(X,Y)Z=\tilde R(X,Y)Z.
\end{equation}
We have
\begin{prop}
The curvature of { an} almost nearly cosymplectic manifold is connected with the curvature of the associated Sasaki-Einsten manifold by
\begin{multline}\label{curv5s}
\bar R(Z,X)Y=
\tilde R(Z,X)Y+\frac12\tilde g(\tilde\phi Z,Y)\tilde\phi X-\frac12\tilde g(\tilde\phi X,Y)\tilde\phi Z+\tilde g(\tilde Z,X)\tilde\phi Y\\-\frac54\tilde\eta(Y)\tilde\eta(Z)X+\frac54\tilde\eta(Y)\tilde\eta(X)Z-\frac12\Big[\tilde\eta(X)\tilde g(Z,Y)-\tilde\eta(Z)\tilde g(X,Y)\Big]\tilde\xi.
\end{multline}
In particular, we have
\begin{equation}\label{rxi}\bar R(Z,X)\bar\xi=\lambda^2\Big[\bar\eta(X)Z-\bar\eta(Z)X\Big].\end{equation}
\end{prop}
\begin{proof}
Applying \eqref{ndhom} to \eqref{curv5} we obtain \eqref{curv5s}. Using the well known identity
\begin{equation}\label{sascurvxi}
R(X,Y)\xi=\eta(Y)X-\eta(X)Y,
\end{equation}
 valid for any Sasakian manifold (see e.g.\cite{Blair}), we obtain \eqref{rxi} from \eqref{curv5s} and \eqref{ndhom}.
\end{proof}

\section{The NGT-connections in dimension 5} Due to Theorem~\ref{asasncos}, any Sasaki Einstein 5-manifold generates an $(R^2-\{0\})$-family of almost nearly cosymplectic structures which are D-homothetic to the Sasaki-Einstein structure. These almost nearly cosymplectic structures have the same Levi-Civita connection. According to Theorem~\ref{acnika} each almost nearly cosymplectic structure generates a unique NGT-connection with  skew-symmetric torsion. The NGTS connections may be different for  different almost nearly cosymplectic structures from the family because the skew-symmetric torsions may be different. We describe these NGTS connections explicitly and express their curvature and the Ricci tensors in terms of the  Sasaki-Einstein curvature and the generated SU(2)-structures.

According to Theorem~\ref{acnika} and { taking} into account Proposition~\ref{ancdf} and \eqref{circ}, the {torsions} $T_2$ , resp $T_1$, of the NGTS connection $\ne$, resp. $\nd$ corresponding to the almost nearly cosymplectic structure $(\bar\eta,\bar\Omega^t_2)$, resp. $(\bar\eta,\bar\Omega^t_1)$  are given respectively by
\begin{eqnarray}\label{ngtor2}
T_2(X,Y,Z)&=&-\frac13d\bar\Omega^t_2(X,Y,Z)=\frac23\lambda(\bar\eta\wedge\bar\Omega^t_1)(X,Y,Z)\\\nonumber&=&-\frac23\cos t\Big(\bar\eta(X)\bar g(\phi hY,Z)+\bar\eta(Y)\bar g(\phi hZ,X)+\bar\eta(Z)\bar g(\phi hX,Y)\Big)\\\nonumber
&+&\frac23\lambda\sin t\Big(\bar\eta(X)\bar g(\phi Y,Z)+\bar\eta(Y)\bar g(\phi Z,X)+\bar\eta(Z)\bar g(\phi X,Y)\Big);\\
T_1(X,Y,Z)&=&-\frac13d\bar\Omega^t_1(X,Y,Z)=-\frac23\lambda\bar\eta\wedge\bar\Omega^t_2(X,Y,Z)\label{ngtor1}\\\nonumber&=&-\frac23\sin t\Big(\bar\eta(X)\bar g(\phi hY,Z)+\bar\eta(Y)\bar g(\phi hZ,X)+\bar\eta(Z)\bar g(\phi hX,Y)\Big)\\\nonumber &-&\frac23\lambda\cos t\Big(\bar\eta(X)\bar g(\phi Y,Z)+\bar\eta(Y)\bar g(\phi Z,X)+\bar\eta(Z)\bar g(\phi X,Y)\Big);
\end{eqnarray}
Letting $t=t+\frac{\pi}2$ we can write all the torsions with one relation as follows
\begin{eqnarray}\label{ngtorr}
T_t(X,Y,Z)&=&-\frac23\sin t\Big(\bar\eta(X)\bar g(\phi hY,Z)+\bar\eta(Y)\bar g(\phi hZ,X)+\bar\eta(Z)\bar g(\phi hX,Y)\Big)\\\nonumber
&-&\frac23\lambda\cos t\Big(\bar\eta(X)\bar g(\phi Y,Z)+\bar\eta(Y)\bar g(\phi Z,X)+\bar\eta(Z)\bar g(\phi X,Y)\Big);
\end{eqnarray}
\subsection{The NGTS connections $\nt$ and its curvature}
Insert \eqref{ngtorr} into \eqref{ngtcon} to get for the NGTS connections $\nt$ the following expression
\begin{multline}\label{ngtcont}
\bar g(\nt_XY,Z)=\bar g(\bar\nabla_XY,Z)+\frac13\Big[\bar\eta(X)\bar g(hY,Z)+\bar\eta(Y)\bar g(hX,Z)\Big]\\
-\frac13\sin t\Big(\bar\eta(X)\bar g(\phi hY,Z)+\bar\eta(Y)\bar g(\phi hZ,X)+\bar\eta(Z)\bar g(\phi hX,Y)\Big)\\
-\frac13\lambda\cos t\Big(\bar\eta(X)\bar g(\phi Y,Z)+\bar\eta(Y)\bar g(\phi Z,X)+\bar\eta(Z)\bar g(\phi X,Y)\Big)
\end{multline}
which, using \eqref{ndhom} and \eqref{con12} yields
\begin{equation}\label{ngtcontt}
\nt_XY=\nabla_XY-\frac{\lambda}2\Big[\eta(X)(\phi^t_2+2\phi_3)Y-\eta(Y)(\phi^t_2-2\phi_3)X+\frac23 g(\phi^t_2X,Y)\xi\Big].
\end{equation}
The equation \eqref{ngtcontt} implies
\begin{prop}
The NGTS connections $\nt$  with totally skew-symmetric torsion on a Sasaki-Einstein 5-manifold $(M,\e,\phi_3,\g,\tilde\xi)$ are related to the Levi-Civita connection $\tilde\nabla$ of $\g$ by
\begin{equation}\label{ngtconttsas}
\nt_XY=\tilde\nabla_XY-\frac12\Big[\e(X)(\phi^t_2+2\phi_3)Y-\e(Y)(\phi^t_2-2\phi_3)X+\frac23 \g(\phi^t_2X,Y)\tilde\xi\Big].
\end{equation}
\end{prop}
Further, to calculate the curvature of the NGTS connection $\nt$ we use \eqref{ngtcontt}. We have
\begin{multline}\label{newt2}
\nt_Z\nt_XY=\G_Z\G_XY-\frac{\lambda}2\Big[\eta(Z)[\phi^t_2+2\phi_3]\G_XY-\eta(\G_XY)[\phi^t_2-2\phi_3]Z+\frac23 g(\phi^t_2Z,\G_XY)\xi\Big]
\\-\frac{\lambda}2\Big\{Z\eta(X)\phi^t_2Y+\eta(X)\nt_Z\phi^t_2Y-Z\eta(Y)\phi^t_2X-\eta(Y)\nt_Z\phi^t_2X+\frac23Zg(\phi^t_2X,Y)\xi +\frac23g(\phi^t_2X,Y)\nt_Z\xi\Big\}\\
-\lambda\Big\{Z\eta(X)\phi_3Y+\eta(X)\nt_Z\phi_3Y+Z\eta(Y)\phi_3X+\eta(Y)\nt_Z\phi_3X\Big\}
\end{multline}
Applying  the fact that $\phi^t_2$ and $\phi^t_1$ are nearly cosymplectic structures, we get from \eqref{newt2} and \eqref{ngtcontt} that
\begin{multline}\label{rngt1}
\frac2{\lambda}R^{ngt}_t(Z,X)Y=\frac2{\lambda}R(Z,X)Y-d\eta(Z,X))[\phi^t_2+2\phi_3]Y\\+(\G_Z\eta)Y)[\phi^t_2-2\phi_3]X-(\G_X\eta)Y)[\phi^t_2-2\phi_3]Z\\
-\eta(X)[(\G_Z\phi^t_2)Y+2(\G_Z\phi_3)Y]+\eta(Z)[(\G_X\phi^t_2)Y+2(\G_X\phi_3)Y]\\-2\eta(Y)[(\G_Z\phi_3)X-(\G_X\phi_3)Z-(\G_Z\phi^t_2)X]+\lambda\Big[\eta(Y)\eta(Z)[\frac52-\frac12\phi^t_1]X] -\eta(Y)\eta(X)[\frac52-\frac12\phi^t_1]Z\Big]\\
+\frac{4\lambda}3[g(\phi^t_2X,Y)\phi_3Z-g(\phi^t_2Z,Y)\phi_3X]-\frac{\lambda}3[g(\phi^t_2X,Y)\phi^t_2Z-g(\phi^t_2Z,Y)\phi^t_2X]\\
-\frac23[g(\G_Z\phi^t_2)X,Y)-g(\G_X\phi^t_2)Z,Y)]\xi+\frac{\lambda}3[\eta(X)g(Z,Y)-\eta(Z)g(X,Y)]\xi\\-\frac{2\lambda}3[\eta(X)g(\phi^t_1Z,Y)-\eta(Z)g(\phi^t_1X,Y)+2\eta(Y)g(\phi^t_1Z,X)]\xi\Big]
\end{multline}
which, due to the Killing condition, \eqref{sasnk}, \eqref{circnc} and \eqref{ngtcont1},  is equivalent to
\begin{multline}\label{rngt2}
\frac2{\lambda}R^{ngt}_t(Z,X)Y=\frac2{\lambda}R(Z,X)Y+2\lambda\omega_3(Z,X))[\phi^t_2+2\phi_3]Y\\-\lambda\omega_3(Z,Y)[\phi^t_2-2\phi_3]X+\lambda\omega_3(X,Y)[\phi^t_2-2\phi_3]Z\\
-\eta(X)[(\G_Z\phi^t_2)Y+2(\G_Z\phi_3)Y]+\eta(Z)[(\G_X\phi^t_2)Y+2(\G_X\phi_3)Y]\\-2\eta(Y)[(\G_Z\phi_3)X-(\G_X\phi_3)Z-(\G_Z\phi^t_2)X]+\lambda\Big[\eta(Y)\eta(Z)[\frac52-\frac12\phi^t_1]X] -\eta(Y)\eta(X)[\frac52-\frac12\phi^t_1]Z\Big]\\
+\frac{4\lambda}3[\Omega^t_2(X,Y)\phi_3Z-\Omega^t_2(Z,Y)\phi_3X]-\frac{\lambda}3[\Omega^t_2(X,Y)\phi^t_2Z-\Omega^t_2(Z,Y)\phi^t_2X]\\
+\frac{\lambda}3[\eta(X)g(Z,Y)-\eta(Z)g(X,Y)]\xi-2\lambda[\eta(X)\Omega^t_1(Z,Y)-\eta(Z)\Omega^t_1(X,Y)]\xi\Big]
\end{multline}
At this point we need  \eqref{ncfxx}, Remark~\ref{rmnk} and \eqref{sasnk} yielding to
\begin{eqnarray}\label{ngtcont1}
g((\nabla_X\phi^t_2)Y,Z)&=&\frac13d\Omega^t_2(X,Y,Z)=-\lambda\eta\wedge\Omega^t_1(X,Y,Z)\\\nonumber&=&-\lambda\Big[\eta(X) g(\phi^t_1Y,Z)+\eta(Y) g(\phi^t_1Z,X)+\eta(Z) g(\phi^t_1X,Y)\Big];\\\label{ngtcont2}
g((\nabla_X\phi^t_1)Y,Z)&=&\frac13d\Omega^t_1(X,Y,Z)=\lambda\eta\wedge\Omega^t_2(X,Y,Z)\\\nonumber&=&\lambda\Big[\eta(X) g(\phi^t_2Y,Z)+\eta(Y) g(\phi^t_2Z,X)+\eta(Z) g(\phi^t_2X,Y)\Big]\\\label{ngtcont3}
g((\nabla_X\phi_3)Y,Z)&=&-\frac1{\lambda}g((\nabla_Xh)Y,Z)=\lambda\Big[ g(X,Y)\eta(Z)- g(X,Z)\eta(Y)\Big].
\end{eqnarray}
Applying \eqref{ngtcont1},  \eqref{ngtcont2}and  \eqref{ngtcont3} to \eqref{rngt2},  we obtain
\begin{multline}\label{rngt4}
\frac2{\lambda^2}R^{ngt}_t(Z,X)Y=\frac2{\lambda^2} R(Z,X)Y+2\omega_3(Z,X)\Big[\phi^t_2+2\phi_3\Big]Y\\
+\frac32\eta(X)\eta(Y)\Big[Z+\phi^t_1Z\Big]+\Big[\omega_3(X,Y)-\frac13\Omega^t_2(X,Y)\Big]\phi^t_2Z-2\Big[\omega_3(X,Y)-\frac23\Omega^t_2(X,Y)\Big]\phi_3Z
\\-\frac32\eta(Z)\eta(Y)\Big[X+\phi^t_1X\Big]-\Big[\omega_3(Z,Y)-\frac13\Omega^t_2(Z,Y)\Big]\phi^t_2X+2\Big[\omega_3(Z,Y)-\frac23\Omega^t_2(Z,Y)\Big]\phi_3X\\
+\Big[\frac53\Big(\eta(Z)g(X,Y)-\eta(X)g(Z,Y)\Big)+\eta(Z)\Omega^t_1(X,Y)-\eta(X)\Omega^t_1(Z,Y)-2\eta(Y)\Omega^t_1(Z,X)\Big]\xi
\end{multline}
The equation \eqref{rngt4}, \eqref{sascurvxi} and the Sasaki-Einstein condition $\tilde Ric=4\g$ imply
\begin{prop}
The curvatures $R^{ngt}_t$ of the  NGTS connections   with totally skew-symmetric torsion on a Sasaki-Einstein 5-manifold $(M,\e,\phi_3,\g,\tilde\xi)$ are related to the Sasaki-Einstein curvature $\tilde R$ and the corresponding SU(2) structure $(\tilde\eta,\O^t_1,\O^t_2,\o_3)$ by
\begin{multline}\label{rngt5}
R^{ngt}_t(Z,X)Y=\tilde R(Z,X)Y+\o_3(Z,X)\Big[\phi^t_2+2\phi_3\Big]Y\\
+\frac34\e(X)\e(Y)\Big[Z+\phi^t_1Z\Big]+\frac12\Big[\o_3(X,Y)-\frac13\O^t_2(X,Y)\Big]\phi^t_2Z-\Big[\o_3(X,Y)-\frac23\O^t_2(X,Y)\Big]\phi_3Z
\\-\frac34\e(Z)\e(Y)\Big[X+\phi^t_1X\Big]-\frac12\Big[\o_3(Z,Y)-\frac13\O^t_2(Z,Y)\Big]\phi^t_2X+\Big[\o_3(Z,Y)-\frac23\O^t_2(Z,Y)\Big]\phi_3X\\
+\Big[\frac56\e(Z)\g(X,Y)-\frac56\e(X)\g(Z,Y)+\frac12\e(Z)\O^t_1(X,Y)-\frac12\e(X)\O^t_1(Z,Y)-\e(Y)\O^t_1(Z,X)\Big]\tilde\xi.
\end{multline}
In particular,
\begin{equation}\label{rngtxi}
R^{ngt}_t(Z,X)\tilde\xi=\frac74\e(X)Z-\frac74\e(Z)X+\frac34\e(X)\phi^t_1Z-\frac34\e(Z)\phi^t_1X-\O^t_1(Z,X)\tilde\xi
\end{equation}
The Ricci tensors of the NGTS connections are given by
\begin{equation}\label{ricngt6}
Ric^{ngt}_t(X,Y)
=\frac53\g(X,Y)+\frac{16}3\e(X)\e(Y)+\frac43\O^t_1(X,Y).
\end{equation}
\end{prop}

{\bf Acknowledgments.} S.I. is partially supported by Contract DH/12/3/12.12.2017 and   Contract 80-10-31/10.04.2019 with the Sofia University
"St.Kl.Ohridski". M.Z. was partially supported by the project
EUROWEB+, and by Serbian Ministry of Education, Science, and
Technological Development, Projects 174012.

\end{document}